\documentclass[12pt]{article}
\usepackage{amsthm, amsmath, amssymb, amsfonts, latexsym}
\date{}
\newcommand{\R}{{\mathbb R}}
\newcommand{\N}{{\mathbb N}}
\newcommand{\const}{{\rm const}}

\newtheorem{theorem}{Theorem}
\newtheorem{corollary}{Corollary}
\newtheorem{lemma}{Lemma}

\theoremstyle{definition}

\newtheorem{remark}{Remark}

\title{On solutions of an ill-posed Stefan problem}
\author{Evgeny Yu. Panov \\ Yaroslav-the-Wise Novgorod State University, \\ Veliky Novgorod, Russian Federation}

\begin{document}
\maketitle

\begin{abstract}
We study multi-phase Stefan problem with increasing Riemann initial data and with generally negative latent specific heats for the phase transitions. We propose the variational formulation of self-similar solutions, which allows to find precise conditions for existence and uniqueness of the solution.
Bibliography$:$
$2$
titles.
\end{abstract}

\section{Introduction}

In a half-plane $\Pi=\{ \ (t,x) \ | \ t>0, x\in\R \ \}$ we consider the multi-phase Stefan problem for the heat equation
\begin{equation}\label{1}
u_t=a_i^2u_{xx}, \quad u_i<u<u_{i+1},
\end{equation}
where $u_-=u_0<u_1<\cdots<u_n<u_{n+1}=u_+$, $u_i$, $i=1,\dots,n$ being the temperatures of phase transitions, $a_i>0$,
$i=0,\dots,n$, are the diffusivity constants. On the unknown lines $x=x_i(t)$ of phase transitions where $u=u_i$ the following Stefan condition
\begin{equation}\label{St}
d_ix_i'(t)+k_iu_x(t,x_i(t)+)-k_{i-1}u_x(t,x_i(t)-)=0
\end{equation}
is postulated, where $k_i>0$ are the thermal conductivity of the $i$-th phase, while $d_i$ is the Stefan number (the latent specific heat) for the $i$-th phase transition. In the case when the temperature $u(t,x)$ non-decreases with respect to the spatial variable $x$ the Stefan numbers $d_i$ should be nonnegative by physical reasons. In this case the problem (\ref{1}), (\ref{St}) is well-posed and reduces to a degenerate nonlinear diffusion  equation,  see \cite[Chapter 5]{LSU}. In the present paper we consider the case of arbitrary $d_i$ when the problem is generally ill-posed. This was firstly demonstrated in \cite{Pukh} for a one-phase Stefan problem with a negative Stefan number. Our aim is to find general conditions which guarantees the correctness of our problem.
We will study the Cauchy problem with the Riemann initial data
\begin{equation}\label{2}
u(0,x)=\left\{\begin{array}{lr} u_-, & x<0, \\ u_+, & x>0. \end{array}\right.
\end{equation}
By the invariance of our problem under the transformation group
$(t,x)\to (\lambda^2 t, \lambda x)$, $\lambda\in\R$, $\lambda\not=0$, it is natural to seek a self-similar solution of problem (\ref{1}), (\ref{St}), (\ref{2}), which has the form $u(t,x)=v(\xi)$, $\xi=x/\sqrt{t}$. For the heat equation
$u_t=a^2 u_{xx}$ a self-similar solution must satisfy the linear ODE $a^2v''=-\xi v'/2$, the general solution of which is
$$
v=C_1F(\xi/a)+C_2, \ C_1,C_2=\const, \mbox{ where } F(\xi)=\frac{1}{2\sqrt{\pi}}\int_{-\infty}^\xi e^{-s^2/4}ds.
$$
This allows to write our solution in the form
\begin{align}\label{3}
v(\xi)=u_i+\frac{u_{i+1}-u_i}{F(\xi_{i+1}/a_i)-F(\xi_i/a_i)}(F(\xi/a_i)-F(\xi_i/a_i)), \\ \nonumber
\xi_i<\xi<\xi_{i+1}, \ i=0,\ldots,n,
\end{align}
where $-\infty=\xi_0<\xi_1<\cdots<\xi_n<\xi_{n+1}=+\infty$ and we agree that $F(-\infty)=0$, $F(+\infty)=1$.
The parabolas $\xi=\xi_i$, $i=1,\ldots,n$, where $u=u_i$, are free boundaries. They must be determined by conditions (\ref{St}).
In the variable $\xi$ these conditions have the form
\begin{equation}\label{4}
d_i\xi_i/2+\frac{k_i(u_{i+1}-u_i)F'(\xi_i/a_i)}{a_i(F(\xi_{i+1}/a_i)-F(\xi_i/a_i))}-
\frac{k_{i-1}(u_i-u_{i-1})F'(\xi_i/a_{i-1})}{a_{i-1}(F(\xi_i/a_{i-1})-F(\xi_{i-1}/a_{i-1}))}=0, \ i=1,\ldots,n.
\end{equation}
To investigate this nonlinear system, we notice that (\ref{4}) coincides with the condition
$\nabla E(\bar\xi)=0$, where the function
\begin{align}\label{5}
E(\bar\xi)=-\sum_{i=0}^n  k_i(u_{i+1}-u_i)\ln (F(\xi_{i+1}/a_i)-F(\xi_i/a_i))+\sum_{i=1}^n d_i\xi_i^2/4, \\ \nonumber \bar\xi=(\xi_1,\ldots,\xi_n)\in\Omega,
\end{align}
the open convex domain $\Omega\subset\R^n$ is given by the inequalities $\xi_1<\cdots<\xi_n$.
Observe that $E(\bar\xi)\in C^\infty(\Omega)$.

\section{Coercivity of $E$}

Let us introduce the sub-level sets $$\Omega_c=\{ \ \bar\xi\in\Omega \ | \ E(\bar\xi)\le c \ \}, \quad c\in\R.$$
We are going to prove that under some exact conditions on the parameters these sets are compact. This means the coercivity of function $E(\bar\xi)$. First, we investigate when the sets $\Omega_c$ are bounded.

\begin{lemma}\label{lem1}
The sets $\Omega_c$ are bounded if and only if the following two conditions are satisfied
\begin{align}\label{coer1}
S^j\doteq\sum_{i=1}^j [\kappa_{i-1}(u_i-u_{i-1})+d_i]\ge 0, \ j=1,\ldots,n; \\
\label{coer2}
S_j\doteq\sum_{i=j}^n [\kappa_i(u_{i+1}-u_i)+d_i]\ge 0, \ j=1,\ldots,n,
\end{align}
where $\kappa_i=k_i/a_i^2$.
\end{lemma}

\begin{proof}
Assume that conditions (\ref{coer1}), (\ref{coer2}) hold. We have to prove that the set $\Omega_c$ is bounded. Assuming the contrary, we can find the sequence $\bar\xi_m\in\Omega_c$, $m\in\N$, such that $|\bar\xi_m|\to +\infty$ as $m\to\infty$ (here and in the sequel we denote by $|v|$ the Euclidean norm of a finite-dimensional vector $v$). Passing to a subsequence if necessary, we can assume that the coordinates $\xi_{mi}$, $i=1,\ldots, n$, of the points $\bar\xi_m$ have finite or infinite limits in $[-\infty,+\infty]$. Let $l$ be the minimal $i$ such that
$\displaystyle\lim_{m\to\infty}\xi_{mi}=+\infty$. If such $i$ does not exist we set $l=n+1$. Similarly, we define $r$ as the maximal $i$ such that $\displaystyle\lim_{m\to\infty}\xi_{mi}=-\infty$. In the case when there is no such $i$ we set $r=0$. By the condition $|\bar\xi_m|\to +\infty$ we can claim that either $l\le n$ or $r\ge 1$. Since the coordinates
$\xi_{mi}$ increase with respect to $i$ then
$$
\lim_{m\to\infty}\xi_{mi}=+\infty \ \forall i, l\le i\le n; \quad \lim_{m\to\infty}\xi_{mi}=-\infty \ \forall i, 1\le i\le r
$$
and, in particular, $r<l$.
Now we observe that by L'H\^{o}pital's rule and the identities $F'(\xi)=\frac{1}{2\sqrt{\pi}}e^{-\xi^2/4}$, $F''(x)=-\xi F'(\xi)/2$,
\begin{align*}
\lim_{\xi\to+\infty}\frac{1-F(\xi)}{2F'(\xi)/\xi}=\lim_{\xi\to+\infty}\frac{F'(\xi)\xi^2}{2(F'(\xi)-\xi F''(\xi))}=\lim_{\xi\to\infty}\frac{\xi^2}{\xi^2+2}=1,\\
\lim_{\xi\to-\infty}\frac{-F(\xi)}{2F'(\xi)/\xi}=\lim_{\xi\to-\infty}\frac{F'(\xi)\xi^2}{2(F'(\xi)-\xi F''(\xi))}=\lim_{\xi\to\infty}\frac{\xi^2}{\xi^2+2}=1.
\end{align*}
By these relations
\begin{equation}\label{ass}
F(-\xi)=1-F(\xi)=\frac{1}{\sqrt{\pi}|\xi|}e^{-\xi^2/4}(1+o(1)) \ \mbox{ as } \xi\to+\infty. \\
\end{equation}
Since
$$
-\ln (F(\xi_{m(i+1)}/a_i)-F(\xi_{mi}/a_i))\ge \max(-\ln (1-F(\xi_{mi}/a_i)),-\ln F(\xi_{m(i+1)}/a_i))\ge 0,
$$
we have
\begin{align}\label{6}
E(\bar\xi_m)=\sum_{i=0}^{r-1} [-k_i(u_{i+1}-u_i)\ln (F(\xi_{m(i+1)}/a_i)-F(\xi_{mi}/a_i))+d_{i+1}\xi_{m(i+1)}^2/4]+\nonumber\\
\sum_{i=l}^n[-k_i(u_{i+1}-u_i)\ln (F(\xi_{m(i+1)}/a_i)-F(\xi_{mi}/a_i))+d_i\xi_{mi}^2/4]-\nonumber\\
\sum_{i=r}^{l-1}k_i(u_{i+1}-u_i)\ln (F(\xi_{m(i+1)}/a_i)-F(\xi_{mi}/a_i))+\sum_{i=r+1}^{l-1}d_i\xi_{mi}^2/4\ge\nonumber\\
\sum_{i=0}^{r-1}[-k_i(u_{i+1}-u_i)\ln F(\xi_{m(i+1)}/a_i)+ d_{i+1}\xi_{m(i+1)}^2/4]+\nonumber\\
\sum_{i=l}^n[-k_i(u_{i+1}-u_i)\ln (1-F(\xi_{mi}/a_i))+d_i\xi_{mi}^2/4]+p_m=\nonumber\\
\sum_{i=1}^r[-k_{i-1}(u_i-u_{i-1})\ln F(\xi_{mi}/a_{i-1})+ d_i\xi_{mi}^2/4]+\nonumber\\
\sum_{i=l}^n[-k_i(u_{i+1}-u_i)\ln (1-F(\xi_{mi}/a_i))+d_i\xi_{mi}^2/4]+p_m,
\end{align}
where
$$
p_m=\sum_{i=r+1}^{l-1}d_i\xi_{mi}^2/4
$$
is a bounded sequence, because the sequences $\xi_{mi}$ have finite limits as $m\to\infty$ for all $i=r+1,\ldots l-1$. Remark that the set of such $i$ may be empty, in this case $p_m=0$. In view of (\ref{ass}), it follows from (\ref{6}) that
\begin{align}\label{7}
c\ge E(\bar\xi_m)\ge \sum_{i=1}^r[\kappa_{i-1}(u_i-u_{i-1})+d_i]\xi_{mi}^2/4+\sum_{i=l}^n [\kappa_i(u_{i+1}-u_i)+d_i]\xi_{mi}^2/4+q_m,
\end{align}
where, up to a vanishing sequence,
$$
q_m=\sum_{i=1}^r k_{i-1}(u_i-u_{i-1}) \ln(\sqrt{\pi}|\xi_{mi}|/a_{i-1})+\sum_{i=l}^n k_i(u_{i+1}-u_i) \ln(\sqrt{\pi}|\xi_{mi}|/a_i)+p_m.
$$
In particular, the sequence $q_m\to+\infty$ as $m\to\infty$ and it follows from (\ref{7}) that
\begin{equation}\label{8}
\sum_{i=1}^r[\kappa_{i-1}(u_{i}-u_{i-1})+d_i]\xi_{mi}^2/4+\sum_{i=l}^n [\kappa_i(u_{i+1}-u_i)+d_i]\xi_{mi}^2/4\le c-q_m<0
\end{equation}
for all sufficiently large $m\in\N$. On the other hand, $\kappa_{i-1}(u_{i}-u_{i-1})+d_i=S^i-S^{i-1}$, $\kappa_i(u_{i+1}-u_i)+d_i=S_i-S_{i+1}$ and applying the summation by parts formula, we find that
\begin{align*}
\sum_{i=1}^r[\kappa_{i-1}(u_{i}-u_{i-1})+d_i]\xi_{mi}^2/4=S^r\xi_{mr}^2/4-\sum_{i=1}^{r-1} S^i(\xi_{m(i+1)}^2-\xi_{mi}^2)/4\ge\\
-\sum_{i=1}^{r-1} S^i(\xi_{m(i+1)}-\xi_{mi})(\xi_{m(i+1)}+\xi_{mi})/4
\ge 0,\\
\sum_{i=l}^n [\kappa_i(u_{i+1}-u_i)+d_i]\xi_{mi}^2/4=S_l\xi_{ml}^2/4+\sum_{i=l+1}^n S_i(\xi_{mi}^2-\xi_{m(i-1)}^2)/4\ge \\
\sum_{i=l+1}^n S_i(\xi_{mi}-\xi_{m(i-1)})(\xi_{mi}+\xi_{m(i-1)})/4\ge 0
\end{align*}
for sufficiently large $m$ because $S_i,S^i\ge 0$ by assumptions (\ref{coer1}), (\ref{coer2}) while $\xi_{mi}<0$  for $i=1,\ldots,r$, $\xi_{mi}>0$ for $i=l,\ldots,n$, if $m$ is sufficiently large. But this contradicts to (\ref{8}). Hence, the set $\Omega_c$ is bounded for every real $c$.

Conversely, assume that at least one of conditions (\ref{coer1}), (\ref{coer2}) fails. For definiteness we suppose that
(\ref{coer1}) fails. This means that for some $1\le r\le n$
\begin{equation}\label{10}
S^r=\sum_{i=1}^r [\kappa_{i-1}(u_i-u_{i-1})+d_i]<0.
\end{equation}
For $\sigma>0$ we define the point $\bar\xi(\sigma)\in\Omega$ with coordinates $\xi_i=-\sigma+i-r$, $i=1,\ldots,r$; $\xi_i=i-r$, $i=r+1,\ldots,n$. Since for $i=1,\ldots,r-1$
$$
F(\xi_{i+1}/a_i)-F(\xi_i/a_i)=\frac{1}{2\sqrt{\pi}}\int_{\xi_i/a_i}^{\xi_{i+1}/a_i}e^{-s^2/4}ds\ge \frac{\xi_{i+1}-\xi_i}{2\sqrt{\pi}a_i}e^{-\xi_i^2/4a_i^2}=\frac{1}{2\sqrt{\pi}a_i}e^{-(\sigma+r-i)^2/4a_i^2},
$$
and, in view of (\ref{ass}),
$$
F(\xi_1/a_0)=\frac{a_0}{\sqrt{\pi}|\sigma+r-1|}e^{-(\sigma+r-1)^2/4a_0^2}(1+o(1))\ \mbox{ as } \sigma\to +\infty
$$
then for some constant $C$
\begin{align}\label{11}
E(\bar\xi(\sigma))\le -\sum_{i=0}^{r-1}k_i(u_{i+1}-u_i)\ln(F(\xi_{i+1}/a_i)-F(\xi_i/a_i))+\nonumber\\
\sum_{i=1}^r d_i\xi_i^2/4+C\le
\sum_{i=1}^r [\kappa_{i-1}(u_i-u_{i-1})+d_i]\sigma^2/4+O(\sigma).
\end{align}
By (\ref{10}) it follows from (\ref{11}) that
$$
E(\bar\xi(\sigma))\to -\infty \ \mbox{ as } \sigma\to+\infty.
$$
Therefore, $\bar\xi(\sigma)\in\Omega_c$ for all sufficiently large $\sigma$ and the set $\Omega_c$ is unbounded.
Thus, if the set $\Omega_c$ is bounded (for some $c$) then both conditions (\ref{coer1}), (\ref{coer2}) are satisfied.
This completes the proof.
\end{proof}

Now, we are going to demonstrate that these conditions are necessary and sufficient for coercivity of the function $E$.

\begin{theorem}\label{th1}
The sets $\Omega_c$ are compact for each $c\in\R$ if and only if conditions (\ref{coer1}), (\ref{coer2}) are satisfied.
In particular, under these conditions the function $E(\bar\xi)$ reaches its minimal value.
\end{theorem}

\begin{proof}
If the set $\Omega_c$ is compact for some $c\in\R$ then it is bounded and by Lemma~\ref{lem1} conditions (\ref{coer1}), (\ref{coer2}) are satisfied.

Conversely, if conditions (\ref{coer1}), (\ref{coer2}) hold then the set $\Omega_c$ is bounded (for every fixed $c$). Therefore, there exists $R>0$ such that $|\bar\xi|\le R$ for all $\bar\xi\in\Omega_c$. This implies the estimate
\begin{align*}
-\sum_{i=0}^n  k_i(u_{i+1}-u_i)\ln (F(\xi_{i+1}/a_i)-F(\xi_i/a_i))\le E(\bar\xi)-\sum_{i=1}^n d_i\xi_i^2/4\le\\
c_1\doteq c+\frac{R^2}{4}\sum_{i=1}^n \max(-d_i,0).
\end{align*}
Since $-k_i(u_{i+1}-u_i)\ln (F(\xi_{i+1}/a_i)-F(\xi_i/a_i))>0$, we claim that for all $i=0,\ldots,n$
$$
-k_i(u_{i+1}-u_i)\ln (F(\xi_{i+1}/a_i)-F(\xi_i/a_i))\le c_1.
$$
It follows that
\begin{equation}\label{12}
F(\xi_{i+1}/a_i)-F(\xi_i/a_i)\ge \delta\doteq\exp(-c_1/\alpha)>0,
\end{equation}
where $\displaystyle\alpha=\min_{i=0,\ldots,n}k_i(u_{i+1}-u_i)>0$. Since $F'(\xi)=\frac{1}{2\sqrt{\pi}}e^{-\xi^2/4}<1$, the function $F(\xi)$ is Lipschitz with constant $1$, and it follows from (\ref{12}) that
$$
(\xi_{i+1}-\xi_i)/a_i\ge F(\xi_{i+1}/a_i)-F(\xi_i/a_i)\ge \delta, \quad i=1,\ldots,n-1,
$$
and we obtain the estimates $\xi_{i+1}-\xi_i\ge\delta_1=\delta\min a_i$. Thus, the set $\Omega_c$ is contained in a compact
$$
K=\{ \ \bar\xi=(\xi_1,\ldots,\xi_n)\in\R^n \ | \ |\bar\xi|\le R, \ \xi_{i+1}-\xi_i\ge\delta_1 \ \forall i=1 ,\ldots,n-1 \ \}.
$$
Since $E(\bar\xi)$ is continuous on $K$, the set $\Omega_c$ where $E(\bar\xi)\le c$ is a closed subset of $K$ and therefore is compact. For $c>N\doteq\inf E(\bar\xi)$, this set is not empty and the entropy $E(\bar\xi)$ reaches on it a minimal value, which is evidently equal $N$.
\end{proof}

We have established the existence of minimal value $E(\bar\xi_0)=\min E(\bar\xi)$. At the point $\bar\xi_0$ the required condition $\nabla E(\bar\xi_0)=0$ is satisfied, and $\bar\xi_0$ is a solution of system (\ref{4}). The coordinates
of $\bar\xi_0$ determine the solution (\ref{3}) of our Stefan-Riemann problem. Thus, we establish the following existence result.

\begin{theorem}\label{th2}
Under conditions (\ref{coer1}), (\ref{coer2}) there exists a self-similar increasing solution of problem (\ref{1}), (\ref{St}), (\ref{2}).
\end{theorem}

\section{Convexity of the Function $E$ and Uniqueness of the Solution}

In this section we find a condition of strict convexity of the function $E(\bar\xi)$. Since a strictly convex function can have at most one critical point (and it is necessarily a global minimum) the system (\ref{4}) has at most one solution, that is, a self-similar solution (\ref{3}) of problem (\ref{1}), (\ref{St}), (\ref{2}) is unique. We introduce the functions
$$
E_i(\bar\xi)=-k_i(u_{i+1}-u_i)\ln (F(\xi_{i+1}/a_i)-F(\xi_i/a_i)), \quad i=0,\ldots,n,
$$
which are exactly summands of the first sum in expression (\ref{5}). They depend only on two coordinates $\xi_i,\xi_{i+1}$ if
$i=1,\ldots,n-1$ or on one coordinate $\xi_1$, $\xi_n$ if $i=0,n$, respectively. We compute the second derivatives
$$
\frac{\partial^2 E_i(\bar\xi)}{\partial\xi_i^2}=\kappa_i(u_{i+1}-u_i)\left(\frac{F''(y_i)}{F(x_i)-F(y_i)}+\frac{(F'(y_i))^2}{(F(x_i)-F(y_i))^2}\right),
$$
where we denote $x_i=\xi_{i+1}/a_i$, $y_i=\xi_i/a_i$.
Since $F''(x)=-\frac{x}{2}F'(x)$ we can transform the above relation
\begin{align}\label{13}
\frac{\partial^2 E_i(\bar\xi)}{\partial\xi_i^2}=\frac{\kappa_i(u_{i+1}-u_i)F'(y_i)}{F(x_i)-F(y_i)}
\left(-\frac{y_i}{2}-\frac{F'(x_i)-F'(y_i)}{F(x_i)-F(y_i)}\right)+\nonumber\\ \frac{\kappa_i(u_{i+1}-u_i)F'(y_i)F'(x_i)}{(F(x_i)-F(y_i))^2}=\beta_i^-+\gamma_i,
\end{align}
where
\begin{align*}
\beta_i^-=\frac{\kappa_i(u_{i+1}-u_i)F'(y_i)}{F(x_i)-F(y_i)}
\left(-\frac{y_i}{2}-\frac{F'(x_i)-F'(y_i)}{F(x_i)-F(y_i)}\right), \ i=1,\ldots,n; \\ \gamma_i=\frac{\kappa_i(u_{i+1}-u_i)F'(y_i)F'(x_i)}{(F(x_i)-F(y_i))^2}, \ i=0,\ldots,n.
\end{align*}
Notice that $x_n=+\infty$, $y_0=-\infty$ and therefore $F(x_n)=1$, $F(y_0)=0$, $F'(y_0)=F'(x_n)=0$. In particular,
$\gamma_0=\gamma_n=0$. By the similar computations we find
\begin{equation}\label{14}
\frac{\partial^2 E_i(\bar\xi)}{\partial\xi_{i+1}^2}=\beta_i^++\gamma_i, \quad \frac{\partial^2 E_i(\bar\xi)}{\partial\xi_i\partial\xi_{i+1}}=-\gamma_i,
\end{equation}
$$
\mbox{ where } \ \beta_i^+=\frac{\kappa_i(u_{i+1}-u_i)F'(x_i)}{F(x_i)-F(y_i)}
\left(\frac{x_i}{2}+\frac{F'(x_i)-F'(y_i)}{F(x_i)-F(y_i)}\right), \quad i=0,\ldots,n-1.
$$
To estimate the values $\beta_i^\pm$, we use the following simple lemma.

\begin{lemma}\label{lem2}
For each $x>y$ there exist such $\theta_1,\theta_2\in (y,x)$ that
\begin{align}\label{15a}
\frac{1}{F(x)-F(y)}\left(-\frac{y}{2}-\frac{F'(x)-F'(y)}{F(x)-F(y)}\right)=\frac{1}{4F'(\theta_1)}, \\
\label{15b}
\frac{1}{F(x)-F(y)}\left(\frac{x}{2}+\frac{F'(x)-F'(y)}{F(x)-F(y)}\right)=\frac{1}{4F'(\theta_2)}.
\end{align}
\end{lemma}

\begin{proof}
Observe that the left hand side of (\ref{15a}) can be written in the form
$$
\frac{-\frac{y}{2}(F(x)-F(y))-(F'(x)-F'(y))}{(F(x)-F(y))^2}=\frac{g(x)-g(y)}{h(x)-h(y)},
$$
where
$$
g(y)=-\frac{y}{2}(F(x)-F(y))-(F'(x)-F'(y)), \quad h(y)=(F(x)-F(y))^2.
$$
By the Cauchy mean value theorem there exists such a value $\theta_1\in (y,x)$ that
$$
\frac{g(x)-g(y)}{h(x)-h(y)}=\frac{g'(\theta_1)}{h'(\theta_1)}.
$$
Using again the identity $F''(y)=-y/2F'(y)$, we find that
$$
g'(y)=-(F(x)-F(y))/2+yF'(y)/2+F''(y)=-(F(x)-F(y))/2, \quad h'(y)=-2(F(x)-F(y))F'(y).
$$
Therefore, $g'(\theta_1)/h'(\theta_1)=1/(4F'(\theta_1))$ and (\ref{15a}) follows.
Similarly, the left hand side of (\ref{15b}) can be represented as the ratio $(g_1(x)-g_1(y))/(h(x)-h(y))$, where
$$
g_1(x)=\frac{x}{2}(F(x)-F(y))+(F'(x)-F'(y)), \quad h(x)=(F(x)-F(y))^2.
$$
Applying again the Cauchy mean value theorem, we arrive at (\ref{15b}). Notice also that (\ref{15b}) reduce to (\ref{15a}) after the change $x'=-y$, $y'=-x$.
\end{proof}

\begin{remark}\label{rem1}
As is easy to realize from the proof, relations (\ref{15a}), (\ref{15b}) remain valid, respectively, for $x=+\infty$, $y=-\infty$.
\end{remark}

\begin{corollary}\label{cor1}
There exist values $\theta_{1i},\theta_{2i}\in (y_i,x_i)$ such that
\begin{equation}\label{16}
\beta_i^-=\kappa_i(u_{i+1}-u_i)F'(y_i)/(4F'(\theta_{1i})), \quad \beta_i^+=\kappa_i(u_{i+1}-u_i)F'(x_i)/(4F'(\theta_{2i})).
\end{equation}
In particular, $\beta_i^\pm>0$ and for $i=1,\ldots,n$
\begin{equation}\label{17}
\beta_i^-+\beta_{i-1}^+ >\frac{1}{4}\min (\kappa_i(u_{i+1}-u_i),\kappa_{i-1}(u_i-u_{i-1})).
\end{equation}
\end{corollary}

\begin{proof}
Inequalities (\ref{16}) directly follows from Lemma~\ref{lem1} and the definitions of $\beta_i^\pm$.
Further, if $\xi_i\ge 0$ then $\theta_{1i}>y_i=\xi_i/a_i\ge 0$, which implies that $F'(y_i)/F'(\theta_{1i})>1$. Therefore,
$\beta_i^->\kappa_i(u_{i+1}-u_i)/4$ and (\ref{17}) follows. Similarly, in the case $\xi_i\le 0$ we have $\theta_{2(i-1)}<x_{i-1}=\xi_i/a_{i-1}\le 0$ and therefore $F'(x_{i-1})/F'(\theta_{2(i-1)})>1$. Hence
$\beta_{i-1}^+ > \kappa_{i-1}(u_i-u_{i-1})/4$ and (\ref{17}) is again fulfilled.
\end{proof}

Now we are ready to formulate a sufficient condition for convexity of the function $E$.

\begin{theorem}\label{th3}
Suppose that for all $i=1,\ldots,n$
\begin{equation}\label{18}
\min (\kappa_i(u_{i+1}-u_i),\kappa_{i-1}(u_i-u_{i-1}))+2d_i\ge 0.
\end{equation}
Then the function $E(\bar\xi)$ is strictly convex on $\Omega$.
\end{theorem}

\begin{proof}
We fix a point $\bar\xi_0=(\xi_{01},\ldots,\xi_{0n})\in\Omega$ and define values $d_i^-$, $d_i^+$ such that $d_i^-+d_i^+=d_i$ and
that
\begin{equation}\label{18a}
\beta_i^-+d_i^-/2>0, \quad \beta_{i-1}^+ +d_i^+/2>0,
\end{equation}
where $\beta_i^\pm$ correspond to $\bar\xi_0$. Since, in view of (\ref{17}), (\ref{18})
$$
\beta_i^-+\beta_{i-1}^++d_i/2>0,
$$
such values really exist, for instance we can take $$d_i^-=\frac{d_i\beta_i^-}{\beta_i^-+\beta_{i-1}^+}, \quad d_i^+=\frac{d_i\beta_{i-1}^+}{\beta_i^-+\beta_{i-1}^+}.$$
We introduce the functions
\begin{align*}
P_i(\bar\xi)=E_i(\xi_i,\xi_{i+1})+d_i^-\xi_i^2/4+d_{i+1}^+\xi_{i+1}^2/4, \ i=1,\ldots,n-1; \\
P_0(\bar\xi)=E_0(\xi_1)+d_1^+\xi_1^2/4, \quad P_n(\bar\xi)=E_n(\xi_n)+d_n^-\xi_n^2/4.
\end{align*}
Observe that
$$
\sum_{i=0}^n P_i(\bar\xi)=\sum_{i=0}^n E_i(\bar\xi)+\sum_{i=1}^n (d_i^-+d_i^+)\xi_i^2/4=\sum_{i=0}^n E_i(\bar\xi)+\sum_{i=1}^n d_i\xi_i^2/4=E(\bar\xi).
$$
First, we observe that by relations (\ref{13}), (\ref{14}), (\ref{18a})
\begin{equation}\label{19}
\frac{\partial^2 P_0(\bar\xi_0)}{\partial\xi_1^2}=P_0''(\xi_{01})=\beta_0^++d_1^+/2>0, \quad \frac{\partial^2 P_n(\bar\xi_0)}{\partial\xi_n^2}=P_n''(\xi_{0n})=\beta_n^-+d_n/2>0.
\end{equation}
For $i=1,\ldots,n-1$ the function $P_i$ depend only on two variables $\xi_i,\xi_{i+1}$, $P_i(\bar\xi)=P_i(\xi_i,\xi_{i+1})$. Denote by $D_2^2 P_i$ the Hessian of $P_i$ as a function of two variables. In view of (\ref{13}), (\ref{14}) we find
$$
\frac{\partial^2 P_i(\bar\xi_0)}{\partial\xi_i^2}=\beta_i^-+d_i^-/2+\gamma_i, \ \frac{\partial^2 P_i(\bar\xi_0)}{\partial\xi_{i+1}^2}=\beta_i^++d_{i+1}^+/2+\gamma_i, \ \frac{\partial^2 P_i(\bar\xi_0)}{\partial\xi_i\partial\xi_{i+1}}=-\gamma_i,
$$
and the matrix $D_2^2 P_i$ can be represented as the sum $R_1+\gamma_iR_2$, where $R_1$ is a diagonal matrix with the positive (in view of (\ref{18a})) diagonal elements
$\beta_i^-+d_i^-/2$, $\beta_i^++d_{i+1}^+/2$ while $R_2=\left(\begin{smallmatrix} 1 & -1 \\ -1 & 1\end{smallmatrix}\right)$. Since $R_1>0$, $R_2\ge 0$, then the matrix $D_2^2 P_i>0$ (strictly positive definite).
This, together with (\ref{19}), implies that the ``full`` Hessian $D^2 P_i(\bar\xi_0)\ge 0$ for all $i=0,\ldots,n$ and therefore
$$
D^2 E(\bar\xi_0)=\sum_{i=0}^n D^2P_i(\bar\xi_0)\ge 0.
$$
Let us prove that $D^2 E(\bar\xi_0)>0$. Assume that for some $\zeta=(\zeta_1,\ldots,\zeta_n)\in\R^n$
\begin{equation}\label{form}
D^2 E(\bar\xi_0)\zeta\cdot\zeta=\sum_{i,j=1}^n\frac{\partial^2 E(\bar\xi_0)}{\partial\xi_i\partial\xi_j}\zeta_i\zeta_j=0.
\end{equation}
Since $D^2 E(\bar\xi_0)$ is a sum of nonnegative matrices $D^2P_i(\bar\xi_0)$, we find that
\begin{equation}\label{20}
D^2 P_i(\bar\xi_0)\zeta\cdot\zeta=0 \quad \forall i=0,\ldots,n.
\end{equation}
Taking $i=0,n$ we derive with the help of (\ref{19}) that $\zeta_1=\zeta_n=0$. In the case $i=1,\ldots,n-1$
(\ref{20}) reduces to the equality
$$
D_2^2 P_i(\bar\xi_0)v_i\cdot v_i=0, \quad v_i\doteq (\zeta_i,\zeta_{i+1})\in\R^2.
$$
Since $D_2^2 P_i(\bar\xi_0)>0$, this implies that $v_i=0$. In particular, $\zeta_i=0$ and we claim that $\zeta=0$.
Thus, (\ref{20}) may hold only for $\zeta=0$, that is, the Hessian $D^2 E(\bar\xi_0)>0$. Taking into account that a point
$\bar\xi_0\in\Omega$ is arbitrary, we conclude that the function $E(\bar\xi)$ is strictly convex.
\end{proof}
\begin{theorem}\label{th4}
Under requirement (\ref{18}) there exists a unique increasing self-similar solution (\ref{3}) of problem (\ref{1}), (\ref{St}), (\ref{2}).
\end{theorem}

\begin{proof}
Obviously, condition (\ref{18}) implies that
\begin{align*}
\kappa_i(u_{i+1}-u_i)+d_i=\kappa_i(u_{i+1}-u_i)/2+(\kappa_i(u_{i+1}-u_i)+2d_i)/2>0, \\ \kappa_{i-1}(u_i-u_{i-1})+d_i=\kappa_{i-1}(u_i-u_{i-1})/2+(\kappa_{i-1}(u_i-u_{i-1})+2d_i)/2>0
\end{align*}
and, in particular, conditions (\ref{coer1}), (\ref{coer2}) are satisfied. By Theorem~\ref{th2} solution (\ref{3}) of problem (\ref{1}), (\ref{St}), (\ref{2}) exists. Uniqueness of this solution follows from strict convexity of the function $E(\bar\xi)$ claimed in Theorem~\ref{th3}.
\end{proof}

\section*{Acknowledgments}
The research was supported by the Russian Science Foundation, grant 22-21-00344.


\begin{thebibliography}{999}

\fontsize{10.6pt}{4mm}{\selectfont

\vskip4pt
\bibitem{LSU} O. A. Ladyzhenskaya, V. A. Solonnikov and N. N. Ural'tseva,
\textit{Linear and Quasi-Linear Equations of Parabolic Type}, AMS, Providence (1968).

\vskip4pt
\bibitem{Pukh}
V. V. Pukhnachov, Generation of a singularity in the solution of a Stefan-type problem [in Russian], \textit{Differ. Uravn.}, \textbf{16}, No. 3, 492--500 (1980).

}
\end{thebibliography}
\end{document}